\documentclass[10pt]{article}

\usepackage[
	a4paper,
	margin=1in
]{geometry}
\usepackage{setspace}
\usepackage{sectsty} 
\usepackage{titlesec}
\usepackage{appendix}

\usepackage{amsmath, amssymb, amsthm}
\usepackage{mathtools}

\usepackage[
	setpagesize=false,
	colorlinks=true,
	linkcolor=BrickRed,
        citecolor=OliveGreen,
        urlcolor=black,
	pdfencoding=auto,
	psdextra,
]{hyperref}
\usepackage{cleveref}

\usepackage{etoolbox} 
\usepackage[shortlabels]{enumitem}
\usepackage{xparse} 

\usepackage{graphicx}
\usepackage[dvipsnames]{xcolor}
\usepackage{tikz}

\usepackage{todonotes}
\usepackage[
    deletedmarkup=sout,
    commentmarkup=uwave,
    authormarkuptext=name
]{changes}

\usepackage{comment}

\setlist[enumerate,1]{label=(\arabic*), ref=(\arabic*)}
\setlist[enumerate,3]{label=(\roman*), ref=(\roman*)}

\allsectionsfont{\boldmath} 

\titlelabel{\thetitle.\quad}

\theoremstyle{plain}
\newtheorem{theorem}{Theorem}[section]
\newtheorem{lemma}[theorem]{Lemma}

\newtheorem{proposition}[theorem]{Proposition}
\newtheorem{observation}[theorem]{Observation}
\newtheorem{conjecture}[theorem]{Conjecture}

\newtheorem{question}[theorem]{Question}
\newtheorem{fact}[theorem]{Fact}

\newtheorem{claim}[theorem]{Claim}
\newtheorem*{claim*}{Claim}

\makeatletter
\newenvironment{claimproof}[1][Proof]{\par
	\pushQED{\qed}%
	
	\normalfont \topsep6\p@\@plus6\p@\relax
	\trivlist
	\item[\hskip\labelsep
	\textit{#1}\@addpunct{.}~]\ignorespaces
}{%
	\popQED\endtrivlist\@endpefalse
}
\makeatother

\theoremstyle{definition}
\newtheorem{definition}[theorem]{Definition}
\newtheorem*{definition*}{Definition}

\newcommand{\ve}{\varepsilon}

\newcommand{\Hom}{\mathrm{Hom}}

\newcommand{\knrs}{\mathrm{KNRS}}

\newcommand{\defeq}{\coloneqq}

\usetikzlibrary{matrix,arrows,decorations.pathmorphing}
\usetikzlibrary{positioning}
\usetikzlibrary{graphs} 
\usetikzlibrary{graphs.standard}

\newcounter{propcounter}
\stepcounter{propcounter}

\title{Counting cographs in locally dense graphs}

\author{Hyunwoo Lee\thanks{{Department of Mathematical Sciences, KAIST, South Korea and Extremal Combinatorics and Probability Group (ECOPRO), Institute for Basic Science (IBS).
        E-mail: {\ttfamily hyunwoo.lee@kaist.ac.kr.} Supported by the National Research Foundation of Korea (NRF) grant funded by the Korea government (MSIT) No. RS-2023-00210430, and the Institute for Basic Science (IBS-R029-C4).}}
}

\usepackage[square,sort,comma,numbers]{natbib}
\begin{document}
\maketitle

\begin{abstract}
    An $n$-vertex graph $G$ is $(\rho, d)$-dense if for all $U \subseteq V(G)$ of size at least $\rho n$, the induced subgraph $G[U]$ contains at least $\frac{d}{2}|U|^2$ edges. In 2010, partly motivated by the theory of quasirandomness, Kohayakawa, Nagle, R\"{o}dl, and Schacht conjectured that, for any fixed graph $F$, the number of homomorphic copies of $F$ in a given $n$-vertex $(o(1),d)$-dense graph is at least $(1-o(1))d^{e(F)}n^{v(F)}$. Despite considerable effort, only a few graphs are known to satisfy the conjecture. We find a new broad class of graphs that satisfy the conjecture, including cographs.
    
\end{abstract}


\section{Introduction}\label{sec:intro}

Counting copies of a fixed graph $F$ in a host graph $G$ is one of the
central problems in extremal combinatorics.  Many classical results can be
interpreted as saying that suitable density assumptions on $G$ force many
copies of $F$.  A convenient language for such questions is that of
homomorphism densities.  For graphs $F$ and $G$, let
\[
    q(F,G) := \frac{|\Hom(F,G)|}{|V(G)|^{|V(F)|}},
\]
which denotes the probability that a uniformly random map $V(F)\to V(G)$ is a graph homomorphism. We also write $t(F,G) \defeq |\Hom(F,G)|$ for the homomorphism count.

In the Erd\H{o}s--R\'enyi random graph $G(n,p)$, one expects
$q(F,G(n,p)) = p^{e(F)}+o(1)$ with high probability for every fixed graph
$F$.  A classical theorem of Chung, Graham, and Wilson~\cite{Chung-Graham-Wilson}
shows that this random-like counting behaviour is equivalent to uniform
control of edge densities on vertex subsets. In one of its standard forms,
their theorem states that if an $n$-vertex graph $G$ has edge density $p$ and
\[
    e_G(U) = \frac{p}{2}|U|^2 \pm o(n^2)
\]
for every $U\subseteq V(G)$, then
\[
    q(F,G) = p^{e(F)}+o(1)
\]
for every fixed graph $F$.  In other words, a two-sided quasirandom control on all vertex subsets forces the same asymptotic homomorphism densities as in the
random graph.

The Kohayakawa--Nagle--R\"odl--Schacht conjecture asks what remains true
when this quasirandomness assumption is weakened in two ways: one only keeps
the lower bound, and one only imposes it on sufficiently large vertex
subsets. For $0<\rho,d<1$, an $n$-vertex graph $G$ is called
$(\rho,d)$-dense if every set $X\subseteq V(G)$ with $|X|\ge \rho n$ spans at
least $\frac{d}{2}|X|^2$ edges.
\begin{conjecture}[Kohayakawa, Nagle, R\"{o}dl, and Schacht~\cite{KNRS}]\label{conj:KNRS}
    For every graph $F$ and $d, \ve > 0$, there exists a positive real number $\rho$ such that the following inequality 
    \begin{equation}
        q(F, G) \geq d^{e(F)} - \ve\label{eq:hom-ineuqality}
    \end{equation}
    holds for all $(\rho, d)$-dense graphs $G$.
\end{conjecture}

We say that a graph $F$ is \emph{KNRS} if, for every $0<d<1$ and every
$\varepsilon>0$, there exists $\rho>0$ such that every sufficiently large
$(\rho,d)$-dense graph $G$ satisfies \eqref{eq:hom-ineuqality}.
The conjecture of Kohayakawa, Nagle, R\"odl, and Schacht then asserts that every
graph is KNRS.

The $(\rho,d)$-denseness condition is also a useful tool in counting
arguments in Ramsey theory.  In their work on Ramsey numbers of sparse graphs, Graham, R\"{o}dl, and Ruci\'{n}ski~\cite{Graham-Rodl-Rucinski} used a dichotomy in which one colour is dense on a suitable subset, or the other colour is dense on all sufficiently large subsets.  Thus embedding problems in locally dense graphs arise naturally in Ramsey-theoretic arguments, making the KNRS conjecture relevant beyond quasirandom counting alone.

Despite its importance, only rather few graphs are known to satisfy the $\knrs$ conjecture. Known examples include complete multipartite
graphs~\cite{KNRS}, bipartite graphs satisfying Sidorenko's
conjecture~\cite{Sidorenko1,Sidorenko2}, odd cycles~\cite{Reiher}, unicyclic graphs~\cite{Joonkyung}, cycles with a chord~\cite{Joonkyung}, and graphs obtained from certain tree-like gluing operations~\cite{Joonkyung}.  Recent progress includes further gluing results of Brada\v{c}, Sudakov, and Wigderson~\cite{Bradac-Sudakov-Wigderson}, as well as subdivision results of Chen, Lin, and Ma~\cite{Chen-Lin-Ma}. See also~\cite{Kral-Volec-Wei}.

We add a large number of graphs to the family of known KNRS graphs by showing that cographs, i.e., induced $P_4$-free graphs, are $\knrs$.  Recall that cographs are equivalently the graphs obtained from a single vertex by repeatedly applying disjoint union and join operations (See~\cite{Graph-class-survey}). This recursive description makes them structurally rich while still highly tractable. It is also worth noting that, while sufficiently broad, the family is easily verifiable algorithmically.

\begin{theorem}\label{thm:main-cograph}
Every cograph is $\knrs$.  That is, for every cograph $F$, every $0<d<1$,
and every $\eta>0$, there exists $\rho>0$ such that every sufficiently
large $(\rho,d)$-dense graph $G$ satisfies
\[
    q(F,G) \ge (1-\eta)d^{e(F)}.
\]
\end{theorem}

Theorem~\ref{thm:main-cograph} follows from a more general closure property of $\knrs$ graphs. To state this result, we introduce the following construction. For later use, we state the construction for $k$-uniform hypergraphs, which we call $k$-graphs.

\begin{definition}
    Let $F_1, F_2$ be $k$-graphs and $F_{uv}$ be a $k$-graph with two labeled vertices $u$ and $v$. Define $J(F_1, F_2; F_{uv})$ as a $k$-graph which is obtained from disjoint copies of $F_1$ and $F_2$ and joining each vertex $x\in V(F_1)$ and $y\in V(F_2)$ by $F_{uv}$ (disjointly) by identifying $u$ and $x$, and $v$ and $y$. 
\end{definition}
Then we have the following theorem, which also implies Theorem~\ref{thm:main-cograph} by analyzing structural properties of cographs.

\begin{theorem}\label{thm:main-graph}
    Let $F_1$, $F_2$ be graphs that are $\knrs$ and $F_{uv}$ be a labelled $\knrs$ graph. Then $J(F_1, F_2; F_{uv})$ is $\knrs$.
\end{theorem}


\section{Extensions to hypergraphs}\label{sec:hypergraphs}

The main idea of the proof of Theorem~\ref{thm:main-graph} is to generalize Conjecture~\ref{conj:KNRS} to hypergraphs, thereby finding symmetries that are absent in the original graph and then analysing such hypergraphs. To achieve this, we define $(\rho, d)$-dense hypergraphs. For convenience in the proof, we formulate both $\knrs$ and $(\rho, d)$-dense hypergraphs in a weighted sense. We assume that $k \geq 2$.

We call a function $H: E(K_n^{(k)}) \to [0, 1]$ a weighted $k$-graph on $n$ vertices. Then for a given $k$-graph $F$, we define
$$
    t(F, H) \defeq \sum_{\phi\in \mathrm{Hom}(F, K_n^{(k)})} \prod_{e\in E(F)} H(\phi(e)).
$$
For a vertex subset $U\subseteq V(H)$, we define a weighted induced subgraph $H[U]$ as a restriction of $H$ to $E(K_n^{(k)}[U])$ with $V(H[U]) \defeq U$.

We say a weighted $k$-graph $H$ is \emph{$(\rho, d)$-dense} if for every vertex subset $U\subseteq V(H)$ of size at least $\rho v(H)$, we have 
$$
    t(K_k^{(k)}, H[U]) \geq d|U|^k.
$$

We raise a question about which hypergraphs satisfy the natural hypergraph variant of Conjecture~\ref{conj:KNRS}.

\begin{question}\label{ques:hypergraph}
    Determine whether a given $k$-graph $F$ satisfies the following property:
    For all real numbers $d, \ve > 0$, there exists a positive real number $\rho$ such that 
    \begin{equation}
        t(F, H) \geq (d^{e(F)} - \ve)n^{v(F)}\label{eq:hom-ineuqality-hypergraph}
    \end{equation}
    holds for every $(\rho, d)$-dense weighted $k$-graph $H$.
    If $F$ satisfies this property, then we say $F$ is $\knrs$, as in the graph case.
\end{question}

The weighted and unweighted versions of the above definitions and statements are equivalent. In the graph case, we refer the reader to~\cite[Corollary 2.12]{JLee-thesis}. Since the proof for hypergraphs is straightforward, and the weighted formulation is arguably more natural and slightly stronger than the unweighted one, we omit the details.

Keeping the hypergraph extension of the notion of $\knrs$ in our mind, we indeed proved a stronger result than Theorem~\ref{thm:main-graph} as follows.

\begin{theorem}\label{thm:main-hypergraph}
    Let $F_1$, $F_2$ be $k$-graphs that are $\knrs$ and $F_{uv}$ be a labelled $\knrs$ $k$-graph. Then $J(F_1, F_2; F_{uv})$ is $\knrs$.
\end{theorem}

Although one may wonder why we consider an extension to hypergraphs, the notion of $\knrs$ for weighted hypergraphs naturally arises in the proof of Theorem~\ref{thm:main-graph}. This is precisely why we introduce the hypergraph setting and establish the more general statement, Theorem~\ref{thm:main-hypergraph}.


\section{Locally dense hypergraphs}\label{sec:lemmas}

Reiher~\cite{Reiher} established a correlation inequality (the case $k=2$ of Lemma~\ref{lem:hypergraph-reiher}) for $(\rho, d)$-dense graphs, which has been instrumental in proving that many graphs with symmetric structures, including odd cycles, are $\knrs$. Our key ingredient in the proof of Theorem~\ref{thm:main-hypergraph} is the hypergraph analogue of Reiher's inequality as follows.

\begin{lemma}[Hypergraph Reiher's inequality]\label{lem:hypergraph-reiher}
    Let $0 < \rho < d \leq 1$ be real numbers and $H$ be an $n$-vertex weighted $(\rho, d)$-dense $k$-graph. Let $f: V(H) \to [0 , 1]$ be a function with $\sum_{v\in V(H)} f(v) \geq \rho n + 1$. Then we have
    \begin{equation}\label{eq:hypergraph-reiher}
        k! \sum_{e\in E(K_n^{(k)})} H(e) \prod_{u_i \in e} f(u_i) \geq d\left( \sum_{v\in V(H)} f(v) \right)^k - 2^k n^{k-1}. 
    \end{equation}
\end{lemma}

The proof of Lemma~\ref{lem:hypergraph-reiher} is given in the Appendix.

Observe that for $d > 0$, the definition of $(\rho, d)$-dense weighted $k$-graph on $n$ vertices implies that for every vertex subset $U$ of size $\lceil \rho n \rceil$, it contains at least one edge of $K_n^{(k)}$. Hence, $\lceil \rho n \rceil \geq k$. This implies the following.

\begin{observation}\label{obs:n-delta}
    If $H$ is a $(\rho, d)$-dense $k$-graph with $0 < \rho < 1$ and $d > 0$, then 
    $$
        v(H) \geq \frac{k-1}{\rho}.
    $$
\end{observation}

As a direct corollary of Lemma~\ref{lem:hypergraph-reiher}, we deduce the refined version of Lemma~\ref{lem:hypergraph-reiher} as follows.

\begin{lemma}\label{lem:hypergraph-reiher-2}
    Let $0 < \rho < d \leq 1$ be real numbers and $H$ be an $n$-vertex weighted $(\rho, d)$-dense $k$-graph. Let $f: V(H) \to [0 , 1]$ be a function. Then we have
    \begin{equation}\label{eq:hypergraph-reiher-2}
        k! \sum_{e\in E(K_n^{(k)})} H(e) \prod_{u_i \in e} f(u_i) \geq d\left( \sum_{v\in V(H)} f(v) \right)^k - 2^{k+1} \rho n^k. 
    \end{equation}
\end{lemma}

\begin{proof}[Proof of Lemma~\ref{lem:hypergraph-reiher-2}]
    We note that $2\rho n \geq \rho n + 1$ by Observation~\ref{obs:n-delta}. 
    If $\sum_{v\in V(H)} f(v) \geq 2 \rho n$, Lemma~\ref{lem:hypergraph-reiher} implies the desired inequality since $2^{k+1}\rho n^k \geq 2^k n^{k-1}$. Hence, we may assume that $\sum_{v\in V(H)} f(v) < 2 \rho n$. Then the right-hand side of \eqref{eq:hypergraph-reiher-2} is at most $d (2\rho n)^k - 2^{k+1} \rho n^k \leq 0$. This completes the proof.
\end{proof}


In the proof of Theorem~\ref{thm:main-hypergraph}, we construct an auxiliary weighted hypergraph from a given host weighted hypergraph by suitably replacing each weighted homomorphism of a fixed hypergraph $F$. For this purpose, we introduce the following definition.

\begin{definition}
    Let $F$ be a $k$-graph and $H: E(K_n^{(k)}) \to [0, 1]$ be a weighted $k$-graph. We define an auxiliary weighted $|V(F)|$-graph $H_{F}: E(K_n^{(v(F))}) \to [0, 1]$ such that for each $e\in E(K_n^{(v(F))})$
    $$
        H_F(e) \defeq \frac{1}{v(F)!} \sum_{\phi\in \mathrm{Iso}(F, K_n^{(k)}[V(e)])} \prod_{f\in E(F)} H(\phi(f)),
    $$
    where $\mathrm{Iso}(F, K_n^{(k)}[V(e)])$ is the set of bijections from $V(F)$ to $V(e)$.
\end{definition}

The following lemma asserts that $H_F$ inherits the $(\rho, d)$-dense property with appropriate parameters if $F$ is $\knrs$.

\begin{lemma}\label{lem:aux-knrs}
    Let $F$ be a $k$-graph which is $\knrs$. Then for every $0 < d, \ve \leq 1$, there exists $\rho > 0$ that depends only on $d, \ve$ and $F$ such that the following holds.
    Let $H$ be an $n$-vertex weighted $(\rho, d)$-dense $k$-graph. Then $H_F$ is a $(\ve, d^{e(F)} - \ve)$-dense weighted $v(F)$-graph.
\end{lemma}

\begin{proof}[Proof of Lemma~\ref{lem:aux-knrs}]
    Let $U \subseteq V(H_F)$ be a vertex subset of size at least $\ve n$.
    Since $F$ is $\knrs$, there exists $\theta > 0$ such that $t(F,L) \geq (d^{e(F)}-\ve/2)v(L)^{v(F)}$ for every weighted $(\theta,d)$-dense $k$-graph $L$. We choose
    $$
        \rho \defeq \min\left\{\ve\theta, \frac{(k-1)\ve^2}{2v(F)^{v(F)+1}}, \frac{\ve}{2}\right\}.
    $$ 
    Note that $\rho < \ve$.
    Then by definition, we have
    \begin{equation}\label{eq:aux-1}
        t(K_{v(F)}^{(v(F))}, H_F[U]) = \sum_{e\in K_n^{(v(F))}[U]} \sum_{\phi\in \mathrm{Iso}(F, K_n^{(k)}[V(e)])} \prod_{f\in E(F)} H(\phi(f)).
    \end{equation}

    Observe that the right-hand side of \eqref{eq:aux-1} is the same as the number of all weighted isomorphic copies of $F$ in $H[U]$. Since the number of degenerated copies of weighted homomorphic copies of $F$ in an $|U|$-vertex $k$-graph is at most $v(F)^{v(F)+1}|U|^{v(F)-1}$, we obtain 
    \begin{equation}\label{eq:aux-2}
        \sum_{e\in K_n^{(v(F))}[U]} \sum_{\phi\in \mathrm{Iso}(F, K_n^{(k)}[V(e)])} \prod_{f\in E(F)} H(\phi(f)) \geq t(F, H[U]) - v(F)^{v(F)+1}|U|^{v(F)-1}.
    \end{equation}

    Since $F$ is $\knrs$ and $H[U]$ is $(\theta, d)$-dense, we deduce 
    \begin{equation}\label{eq:aux-3}
        t(F, H[U]) - v(F)^{v(F)+1}|U|^{v(F)-1} \geq (d^{e(F)}-\ve/2)|U|^{v(F)} - v(F)^{v(F)+1}|U|^{v(F)-1}
    \end{equation}
    provided that $\rho \leq \ve\theta$.

    By combining \eqref{eq:aux-1},\eqref{eq:aux-2}, and \eqref{eq:aux-3}, we obtain
    \begin{equation}\label{eq:aux-4}
        t(K_{v(F)}^{(v(F))}, H_F[U]) \geq \left(d^{e(F)} - \frac{\ve}{2} - \frac{v(F)^{v(F)+1}}{\ve n} \right) |U|^{v(F)}
    \end{equation}
    
    By Observation~\ref{obs:n-delta}, we have
    \begin{equation}\label{eq:aux-5}
        \frac{v(F)^{v(F)+1}}{\ve n} \leq \frac{\ve}{2}.
    \end{equation}

    Hence, by combining \eqref{eq:aux-4} and \eqref{eq:aux-5}, we have
    \begin{equation*}
         t(K_{v(F)}^{(v(F))}, H_F[U]) \geq (d^{e(F)} - \ve)|U|^{v(F)}.
    \end{equation*}
    Thus, $H_F$ is a $(\ve, d^{e(F)} - \ve)$-dense $(v(F))$-graph.
    This completes the proof.
\end{proof}

Lastly, we define the following notion of labelled weighted homomorphism.
\begin{definition}
    Let $F$ be a $k$-graph and $H$ be an $n$-vertex weighted $k$-graph. Let $U \subseteq V(F)$ be a vertex subset and $\psi:U\to V(H)$ be a map. Then we define
    $$
        t(F, H; \psi) \defeq \sum_{\substack{\phi\in \mathrm{hom}(F, K_n^{(k)}) \\: \phi\vert_U = \psi}}\prod_{f\in E(F)} H(\phi(f)),
    $$ and
    $$
        q(F, H; \psi) \defeq \frac{t(F, H; \psi)}{n^{v(F) - |U|}}.
    $$
\end{definition}

For simplicity, we sometimes write a single vertex $v$ as $\{v\}$.
Note that $q(F, H; \psi) \in [0, 1]$, and $t(F,H;\psi)=q(F,H;\psi)=0$ if $\psi$ has no homomorphic extension to $F$.

\begin{lemma}\label{lem:labelled-sum}
    Let $F$ be a $k$-graph and $H$ be an $n$-vertex weighted $k$-graph. Let $U \subseteq V(F)$ be a vertex subset. Then we have
    \begin{equation*}
       0\leq t(F, H) - \sum_{\substack{\psi\in \hom(F[U], K_n^{(k)})\\: |\psi(U)| = |U|}} t(F, H; \psi) \leq (v(F)-1)^{v(F)} n^{v(F)-1}.
    \end{equation*}
\end{lemma}

\begin{proof}[Proof of Lemma~\ref{lem:labelled-sum}]
    We note that the quantity 
    $$
        \sum_{\substack{\psi\in \hom(F[U], K_n^{(k)})\\: |\psi(U)| = |U|}} t(F, H; \psi)
    $$ 
    is at most the number of weighted homomorphic copies of $F$ in $H$ and at least the number of weighted isomorphic copies of $F$ in $H$. Hence, the first inequality is trivial, and the difference between $t(F, H)$ and this quantity is at most the number of degenerated weighted homomorphic copies of $F$ in $H$, which is at most $(v(F)-1)^{v(F)} n^{v(F)-1}$. This completes the proof.
\end{proof}


\section{Proofs of the main results}\label{sec:proof-main}
In this section, we prove Theorem~\ref{thm:main-hypergraph} and Theorem~\ref{thm:main-cograph}.
The proof involves several closely related parameters, whose hierarchy we specify as follows.
When we write $0 < \alpha \ll \beta_1, \beta_2 \ll \gamma < 1$, we mean that there exist functions $f$ and $g$ such that the statement holds whenever $\beta_1, \beta_2 < f(\gamma)$ and $\alpha < g(\beta_1, \beta_2)$. We do not attempt to determine these functions explicitly.

We first prove a special case of Theorem~\ref{thm:main-hypergraph} as a lemma, namely the case where $F_1, F_2$ are disjoint unions of vertices. We denote by $I_{\ell}$ the hypergraph of the disjoint union of $\ell$ vertices.

\begin{lemma}\label{lem:complete-bipartite-like}
    Let $r,s$ be positive integers and $F_{uv}$ be a labelled $\knrs$ $k$-graph. Then $J(I_{r}, I_{s}; F_{uv})$ is $\knrs$.
\end{lemma}

\begin{proof}[Proof of Lemma~\ref{lem:complete-bipartite-like}]
    We fix the hierarchy of parameters as 
    $$
        0 < \rho \ll \mu \ll \eta \ll d, \ve, \frac{1}{v(J(I_r, I_s; F_{uv}))} \leq 1.
    $$
    From Observation~\ref{obs:n-delta}, we may assume that $n$ is sufficiently large.

    Let $H$ be an $n$-vertex weighted $(\rho, d)$-dense $k$-graph. 
    We first claim that $J(I_r, z; F_{uv})$ is $\knrs$, where $z$ is a single vertex graph with vertex set $\{z\}$.

    \begin{claim}\label{clm:z}
        $t(J(I_r, z; F_{uv}), H) \geq (d^{e(J(I_r, z; F_{uv}))} - \eta) n^{v(J(I_r, z; F_{uv}))}$.
    \end{claim}

    \begin{claimproof}[Proof of Claim~\ref{clm:z}]
        By applying Jensen's inequality and the fact that $F_{uv}$ is $\knrs$, we obtain
        \allowdisplaybreaks
        \begin{align}
            t(J(I_r, z; F_{uv}), H) &\geq \sum_{\substack{\psi\in \Hom(z, K_n^{(k)}), \\ |\psi(z)| = 1}} t(J(I_r, z; F_{uv}), H; \psi) \label{eq:fix-1}
            \\&\geq \sum_{\substack{\psi\in \Hom(z, K_n^{(k)}), \\ |\psi(z)| = 1}} \left( t(J(x, z; F_{uv}), H; \psi)\right)^r \nonumber
            \\&\geq n \cdot \left( n^{-1} \sum_{\substack{\psi\in \Hom(z, K_n^{(k)}), \\ |\psi(z)| = 1}}  t(J(x, z; F_{uv}), H; \psi) \right)^r \label{eq:jensen-1}
            \\&\geq n \cdot \left( n^{-1} \cdot \left(t(J(x, z; F_{uv}), H) - \mu n^{v(J(x, z; F_{uv}))}\right) \right)^r \label{eq:fix-2}
            \\&= n^{-r + 1} \cdot \left( t(F_{uv}, H) - \mu n^{v(F_{uv})} \right)^r \nonumber
            \\&\geq n^{-r + 1} \cdot \left( \left(d^{e(F_{uv})} -2\mu \right) n^{v(F_{uv})} \right)^r\label{eq:using-knrs-Fuv}
            \\&\geq (d^{e(J(I_r, z; F_{uv}))} - \eta)n^{v(J(I_r, z; F_{uv}))}. \nonumber
        \end{align}
        The inequalities \eqref{eq:fix-1}, \eqref{eq:fix-2} hold by Lemma~\ref{lem:labelled-sum}, and \eqref{eq:jensen-1} holds by applying Jensen's inequality together with the fact that $x\to x^r$ is a concave upward function. Lastly, \eqref{eq:using-knrs-Fuv} is true because $F_{uv}$ is $\knrs$. This completes the proof.
    \end{claimproof}
    
    Similarly, together with Claim~\ref{clm:z}, we have
    \allowdisplaybreaks
    \begin{align}
        t(J(I_r, I_s; F_{uv}), H) &\geq \sum_{\substack{\psi\in \Hom(I_r, K_n^{(k)}), \\ |\psi(I_r)| = r}} t(J(I_r, I_s; F_{uv}), H; \psi) \label{eq:fix-3}
        \\&\geq \sum_{\substack{\psi\in \Hom(I_r, K_n^{(k)}), \\ |\psi(I_r)| = r}} \left( t(J(I_r, z; F_{uv}), H; \psi)\right)^s \nonumber
        \\&\geq n^r \cdot \left( n^{-r} \sum_{\substack{\psi\in \Hom(I_r, K_n^{(k)}), \\ |\psi(I_r)| = r}}  t(J(I_r, z; F_{uv}), H; \psi) \right)^s \label{eq:jensen-2}
        \\&\geq n^r \cdot \left( n^{-r} \cdot \left(t(J(I_r, z; F_{uv}), H) - \eta n^{v(J(I_r, z; F_{uv}))}\right) \right)^s\label{eq:fix-4}
        \\&\geq n^{-r(s-1)} \cdot \left( \left( d^{e(J(I_r, z;F_{uv}))} - 2\eta \right) n^{v(J(I_r, z; F_{uv}))}\right)^s\label{eq:using-above-clm}
        \\&\geq (d^{e(J(I_r, I_s;F_{uv}))} - \ve) n^{v(J(I_r, I_s; F_{uv}))}. \nonumber
    \end{align}
    The above inequalities hold for exactly the same reason in the proof of Claim~\ref{clm:z}. The inequalities \eqref{eq:fix-3} and \eqref{eq:fix-4} hold by Lemma~\ref{lem:labelled-sum}, and \eqref{eq:jensen-2} is obtained from Jensen's inequality and the fact that $x\to x^s$ is a concave upward function. Lastly, Claim~\ref{clm:z} implies \eqref{eq:using-above-clm} holds.
    This completes the proof.
\end{proof}

We are now ready to prove Theorem~\ref{thm:main-hypergraph}.

\begin{proof}[Proof of Theorem~\ref{thm:main-hypergraph}]
    We fix the hierarchy of parameters as 
    $$
        0 < \rho \ll \ve'' \ll \ve' \ll d, \ve, \frac{1}{v(J(F_1, F_2; F_{uv}))} \leq 1.
    $$
    Again, from Observation~\ref{obs:n-delta}, we may assume that $n$ is sufficiently large.

    Let $r \defeq v(F_1)$ and $s \defeq v(F_2)$. Now let $H$ be an $n$-vertex weighted $(\rho, d)$-dense $k$-graph. By Lemma~\ref{lem:aux-knrs}, both $H_{F_1}$ and $H_{F_2}$ are $(\ve'', d^{e(F_1)} - \ve'')$-dense and $(\ve'', d^{e(F_2)} - \ve'')$-dense, respectively. If $r=1$ or $s=1$, the corresponding auxiliary hypergraph has all vertex weights equal to one, so the corresponding application of Lemma~\ref{lem:hypergraph-reiher-2} below holds directly. 

    For the sake of simplicity, let $F = J(F_1, F_2; F_{uv})$ and $F' = J(F_1, I_s; F_{uv})$.
    We start by proving the following claim.
    
    \begin{claim}\label{clm:one-side}
        $t(F', H) \geq (d^{e(F')} - \ve') n^{v(F')}$.
    \end{claim}

    \begin{claimproof}[Proof of Claim~\ref{clm:one-side}]
        
        For $\psi\in \Hom(I_s, K_n^{(k)})$ with $|\psi(I_s)| = s$ and a vertex $v\in V(K_n^{(k)})$, we define a function $f_{\psi}: V(K_n^{(k)}) \to [0, 1]$ as 
        $$
            f_{\psi}(v) \defeq  q(J(x, I_s, F_{uv}), H; \widehat{\psi}),
        $$
        where $\widehat{\psi}:\{x\}\cup V(I_s)\to V(H)$ is the map such that $\widehat{\psi}(x) = v$ and $\widehat{\psi}\vert_{V(I_s)} = \psi$.

        Then we have
        \allowdisplaybreaks
        \begin{align}
            t(F', H) &\geq  \sum_{\substack{\psi \in \Hom(I_s, K_n^{(k)}), \\ |\psi(I_s)| = s}} \sum_{e\in K_n^{(r)}} r! H_{F_1}(e) n^{(v(F_{uv})-2)rs} \prod_{v\in V(e)} f_{\psi}(v) \nonumber
            \\ &\geq n^{(v(F_{uv}) - 2)rs} \sum_{\substack{\psi \in \Hom(I_s, K_n^{(k)}), \\ |\psi(I_s)| = s}} \left[ (d^{e(F_1)} - \ve'') \left( \sum_{v\in V(K_n^{(k)})} f_{\psi}(v)\right)^r -2^{r+1}\ve'' n^r\right] \label{eq:using-hypergraph-reiher}
            \\ &\geq \left[ (d^{e(F_1)} - \ve'') \sum_{\substack{\psi \in \Hom(I_s, K_n^{(k)}), \\ |\psi(I_s)| = s}} t(J(I_r, I_s; F_{uv}), H; \psi)\right] - 2^{r+1}\ve'' n^{v(F')} \nonumber
            \\ &\geq (d^{e(F_1)} - \ve'')t(J(I_r, I_s; F_{uv}), H) - (2^{r+1}+1)\ve'' n^{v(F')} \label{eq:fix-5}
            \\ &\geq  (d^{e(F_1)} - \ve'')(d^{e(F_{uv})rs} - \ve'') n^{v(F')} - (2^{r+1}+1)\ve'' n^{v(F')} \label{eq:using-bipartite-like}
            \\ &\geq (d^{e(F')} - \ve')n^{v(F')}. \nonumber
        \end{align}
        The inequality \eqref{eq:using-hypergraph-reiher} holds by Lemma~\ref{lem:hypergraph-reiher-2}, and \eqref{eq:fix-5} is obtained from Lemma~\ref{lem:labelled-sum}. Also, we obtain \eqref{eq:using-bipartite-like} by using the fact that $J(I_r, I_s; F_{uv})$ is $\knrs$, which is provided by Lemma~\ref{lem:complete-bipartite-like}.
        This completes the proof.
    \end{claimproof}

    Similar to the proof of Claim~\ref{clm:one-side}, for $\phi\in \Hom(F_1, K_n^{(k)})$ with $|\phi(V(F_1))| = r$ and a vertex $v\in V(K_n^{(k)})$, define $w_{\phi} \defeq \prod_{a\in E(F_1)} H(\phi(a))$ and a function $g_{\phi}: V(K_n^{(k)}) \to [0, 1]$ as 
    $$
        g_{\phi}(v) \defeq  q(J(I_r, y, F_{uv}), H; \widehat{\phi}),
    $$
    where $\widehat{\phi}:\{y\}\cup V(F_1)\to V(H)$ is the map such that $\widehat{\phi}(y) = v$ and $\widehat{\phi}\vert_{V(F_1)} = \phi$, and $I_r$ has vertex set $V(F_1)$.

    Then we have
    \allowdisplaybreaks
    \begin{align}
        t(F, H) &\geq  \sum_{\substack{\phi \in \Hom(F_1, K_n^{(k)}), \\ |\phi(V(F_1))| = r}} w_{\phi} \sum_{e\in K_n^{(s)}} s! H_{F_2}(e) n^{(v(F_{uv})-2)rs} \prod_{v\in V(e)} g_{\phi}(v) \nonumber
        \\ &\geq n^{(v(F_{uv}) - 2)rs} \sum_{\substack{\phi \in \Hom(F_1, K_n^{(k)}), \\ |\phi(V(F_1))| = r}} w_{\phi} \left[ (d^{e(F_2)} - \ve') \left( \sum_{v\in V(K_n^{(k)})} g_{\phi}(v)\right)^s -\ve'n^s\right] \label{eq:using-hypergraph-reiher-2}
        \\ &\geq \left[ (d^{e(F_2)} - \ve') \sum_{\substack{\phi \in \Hom(F_1, K_n^{(k)}), \\ |\phi(V(F_1))| = r}} t(J(F_1, I_s; F_{uv}), H; \phi)\right] - \ve' n^{v(F)}
        \\ &\geq (d^{e(F_2)} - \ve')t(F', H) - 2\ve' n^{v(F)} \label{eq:fix-6}
        \\ &\geq  (d^{e(F_2)} - \ve')(d^{e(F')} - \ve') n^{v(F')} - 2\ve' n^{v(F)} \label{eq:using-one-side}
        \\ &\geq (d^{e(F)} - \ve)n^{v(F)}.
    \end{align}
    Note that Lemma~\ref{lem:hypergraph-reiher-2} implies \eqref{eq:using-hypergraph-reiher-2}, using $2^{s+1}\ve'' \leq \ve'$. Also, $\sum_{\phi}w_{\phi}\leq n^r$, and $n^{(v(F_{uv})-2)rs}w_{\phi}(\sum_v g_{\phi}(v))^s=t(F',H;\phi)$. The inequality \eqref{eq:fix-6} holds by Lemma~\ref{lem:labelled-sum}. Lastly, we obtain \eqref{eq:using-one-side} from Claim~\ref{clm:one-side}.
    This completes the proof.
\end{proof}


We now prove Theorem~\ref{thm:main-cograph}, as a corollary of Theorem~\ref{thm:main-graph}. For this, we need an alternative definition of cographs. 
It is now well-known that the following proposition of cographs is indeed an equivalent definition of a cograph (See~\cite{Graph-class-survey}).

\begin{fact}\label{fact:cograph}
    We say a graph $G$ is a cograph if there is a construction of $G$, starting from a single-vertex graph, using the following rules:
    \begin{enumerate}
        \item[$(a)$] A single-vertex graph is a cograph.
        \item[$(b)$] The disjoint union of two cographs is a cograph.
        \item[$(c)$] The join of two cographs is a cograph. 
    \end{enumerate}
\end{fact}

We note that the \emph{join} of two vertex-disjoint graphs $G_1$ and $G_2$ is the graph obtained by taking their union and adding all edges between $V(G_1)$ and $V(G_2)$. With our definition, the join of $G_1$ and $G_2$ is $J(G_1, G_2; (uv))$, where $(uv)$ is a labelled $K_2$. 

\begin{proof}[Proof of Theorem~\ref{thm:main-cograph}]
    We use induction on the number of vertices. The base case, a single-vertex graph, is trivially $\knrs$. Let $G$ be a cograph with $v(G) \geq 2$. Observe that Fact~\ref{fact:cograph} implies either $G$ is a disjoint union of two cographs or a join of two cographs. Assume $G$ is a disjoint union of two cographs, say $F_1$ and $F_2$. Since $v(F_1), v(F_2) < v(G)$, by the induction hypothesis, both $F_1$ and $F_2$ are $\knrs$. The disjoint union of two $\knrs$ graphs remains $\knrs$, since $t(F_1\sqcup F_2,H)=t(F_1,H)t(F_2,H)$. Hence, in this case, $G$ is $\knrs$.
    Thus, we may assume that $G$ is a join of two cographs $F_1$ and $F_2$. Again, by the induction hypothesis, both $F_1$ and $F_2$ are $\knrs$. Then by Theorem~\ref{thm:main-graph}, we deduce that $J(F_1, F_2; (uv))$, which is $G$, is $\knrs$.
    This completes the proof.
\end{proof}


\section{Concluding remarks}\label{sec:concluding}

In this paper, we found a new class of $\knrs$ graphs, specifically, cographs. Our method considers an analogous notion of $\knrs$ in a hypergraph setting, and uses a hypergraph variant of Reiher's inequality. In this context, we extended Conjecture~\ref{conj:KNRS} to hypergraphs and posed Question~\ref{ques:hypergraph}.  

We note that for all $k \geq 3$, there is a $k$-graph that is not $\knrs$, namely a hypergraph with non-vanishing \emph{uniform Tur\'{a}n density}. For a $k$-graph $F$, its uniform Tur\'{a}n density is defined as the supremum over all $d\in[0,1]$ such that for every $\rho > 0$ and every positive integer $n_0$, there is a $(\rho, d)$-dense $k$-graph on at least $n_0$ vertices that does not contain an isomorphic copy of $F$.

It is known that complete graphs are $\knrs$~\cite{KNRS}. Hence, every sufficiently large locally dense graph contains a copy of any fixed graph, and every graph has vanishing uniform Tur\'{a}n density. However, for all $k\geq 3$, there is a $k$-graph whose uniform Tur\'{a}n density is non-zero, see~\cite{Erdos-Sos,Reiher-Rodl-Schacht}. In this context, Erd\H{o}s and S\'{o}s~\cite{Erdos-Sos} suggest determining the uniform Tur\'{a}n densities of hypergraphs, and it has now become one of the central topics in extremal combinatorics. Observe that from its definition, if a $k$-graph $F$ has a positive uniform Tur\'{a}n density, then $F$ is not $\knrs$. Thus, the natural intuition would be that this observation provides a complete answer to Question~\ref{ques:hypergraph}. However, somewhat surprisingly, the $3$-uniform loose triangle has vanishing uniform Tur\'{a}n density but is not $\knrs$. We include the proof below.  

\begin{proposition}\label{prop:loose-triangle}
    A $3$-uniform loose triangle is not $\knrs$.
\end{proposition}

\begin{proof}[Proof of Proposition~\ref{prop:loose-triangle}]
    Let $T$ be the $3$-uniform loose triangle. By a result of Kostochka, Mubayi and Verstra\"{e}te~\cite[Theorem~1.2]{Kostochka-Mubayi-Verstraete}, there is a sequence of linear $T$-free $3$-graphs $H$ with $q \defeq v(H) \to \infty$ and independence number $a \defeq \alpha(H) \leq q^{2/3+o(1)}$.

    For a positive integer $N$, replace each vertex $i$ of $H$ by a set $V_i$ of $N$ vertices. Let $H_N$ be the $3$-graph consisting of all triples within each $V_i$ and all triples with one vertex in each of $V_i,V_j,V_k$ whenever $\{i,j,k\}\in E(H)$. We write $n \defeq qN$ and $d \defeq 1/(9a^2)$.

    We first show that $H_N$ is $(\rho,d)$-dense for any $\rho>0$ when $N$ is sufficiently large. An independent set of $H_N$ contains at most two vertices from each $V_i$, and the set of parts it meets is independent in $H$. Hence, $\alpha(H_N)\leq 2a$. Every set of $3a$ vertices therefore spans at least $a$ edges, since otherwise deleting one vertex from each edge would leave an independent set of more than $2a$ vertices. For any $U\subseteq V(H_N)$ with $u \defeq |U|\geq 6a$, averaging over all $3a$-subsets of $U$ gives
    $$
        \frac{e(H_N[U])}{\binom{u}{3}} \geq \frac{a}{\binom{3a}{3}} \geq \frac{2}{9a^2}.
    $$
    Since $u(u-1)(u-2)\geq u^3/2$, it follows that $6e(H_N[U])\geq du^3$. Thus, $H_N$ is $(\rho,d)$-dense whenever $\rho qN\geq 6a$.

    We now bound $t(T,H_N)$. Let $S$ be the shadow graph of $H$, and let $M$ be its adjacency matrix. By linearity, $e(S)=3e(H)$. Moreover, every triangle of $S$ comes from an edge of $H$, since otherwise its three pairs extend to three distinct edges of $H$ forming a loose triangle. Hence, $S$ has exactly $e(H)$ triangles. Also, the neighbourhood of every vertex $v$ of $H$ is independent: an edge within this neighbourhood, together with two edges through $v$, would form a loose triangle. Consequently, $2\deg_H(v)\leq a$, and summing over $v$ gives $e(H)\leq qa/6$.

    For fixed parts $V_i$ and $V_j$ containing the two degree-two vertices of an edge of $T$, the number of possible parts for its degree-one vertex is at most $(I+M)_{ij}$. Summing first over the parts containing the three degree-one vertices, and allowing coincidences within edges when counting, we obtain
    \begin{align*}
        \frac{t(T,H_N)}{n^6} &\leq \frac{\operatorname{tr}((I+M)^3)}{q^6}
        \\ &= \frac{q+3\operatorname{tr}(M^2)+\operatorname{tr}(M^3)}{q^6}
        \\ &= \frac{q+24e(H)}{q^6} \leq \frac{1+4a}{q^5}.
    \end{align*}
    Therefore,
    $$
        \frac{t(T,H_N)}{d^3n^6} \leq \frac{9^3a^6(1+4a)}{q^5} \leq q^{-1/3+o(1)} \to 0.
    $$
    Choose $H$ sufficiently large that the last ratio is less than $1/2$, and then fix $H$ and $d$. For every $\rho>0$, taking $N$ sufficiently large gives a $(\rho,d)$-dense $3$-graph $H_N$ with $t(T,H_N)<d^3n^6/2$. Taking $\ve=d^3/4$ shows that $T$ is not $\knrs$. This completes the proof.
\end{proof}


\subsection*{Acknowledgement}
The author is grateful to Joonkyung Lee, Mathias Schacht, and Bjarne Sch\"{u}lke for their helpful discussions and comments over the past two years. He especially thanks Joonkyung Lee for valuable discussions on related topics and helpful comments on the introduction, and Mathias Schacht for suggesting the study of Conjecture~\ref{conj:KNRS}.

\subsection*{Statement of AI use}
The main results of this paper, namely Theorems~\ref{thm:main-cograph},~\ref{thm:main-graph}, and~\ref{thm:main-hypergraph}, were proved before November 2024. The first draft of this paper was completed in April 2026. ChatGPT Pro was used only to find a proof of Proposition~\ref{prop:loose-triangle} and to proofread the manuscript.

\vspace{-0.2cm}

\providecommand{\MR}[1]{}
\providecommand{\MRhref}[2]{%
  \href{http://www.ams.org/mathscinet-getitem?mr=#1}{#2}
}

    \bibliographystyle{amsplain_initials_nobysame}
    \bibliography{bibfile}


\appendix

\section{Appendix}\label{sec:appendix}

We provide the proof of Lemma~\ref{lem:hypergraph-reiher} here.
We adapt the proof method demonstrated in~\cite{Reiher}.

\begin{proof}[Proof of Lemma~\ref{lem:hypergraph-reiher}]
    We note that the space $\{\mathbf{x}\in [0, 1]^{V(H)}: \sum_{v\in V(H)} (\mathbf{x})_v \geq \rho n + 1\}$ is compact, so the following value $\Omega$ defined by
    \begin{equation}\label{eq:minimum}
        \Omega \defeq \min_{\substack{f:V(H)\to [0, 1],\\ \sum_{v\in V(H)} f(v) \geq \rho n + 1}} \left[ k! \sum_{e\in E(K_n^{(k)})} H(e) \prod_{u_i \in e} f(u_i) - d\left( \sum_{v\in V(H)} f(v) \right)^k \right]
    \end{equation}
    exists.
    
    Let $G(f)$ denote the expression in square brackets in \eqref{eq:minimum}.
    Among all the possible choices of $f$ where $G(f) = \Omega$, we choose $f$ that maximizes the size of the set $\{v\in V(H): f(v) \in \{0, 1\}\}$.

    \begin{claim}\label{clm:one-exception}
        There is at most one vertex $v\in V(H)$ such that $f(v) \notin \{0, 1\}$.
    \end{claim}

    \begin{claimproof}[Proof of Claim~\ref{clm:one-exception}]
        Assume there are two distinct vertices $u, v\in V(H)$ such that $0 < f(v) \leq f(u) < 1$. Let $\eta$ be the real number that we will choose later.

        Define a new function $\Tilde{f}: V(H) \to [0, 1]$ such that
        $$
            \Tilde{f}(z) \defeq 
            \begin{cases}
                f(z) + \eta &\text{if } z = u,\\
                f(z) - \eta &\text{if } z = v,\\
                f(z) &\text{otherwise.}
            \end{cases}
        $$

        Then there are four real numbers $A \geq 0$ and $B, C, D$ which do not depend on $\eta$ such that
        \begin{align*}
            G(\Tilde{f}) &= A(f(u)+\eta)(f(v) - \eta) + B(f(u) + \eta) + C(f(v) - \eta) + D\\
            &= G(f) + (A f(v) - A f(u) + B - C)\eta - A \eta^2.
        \end{align*}

        If $A \neq 0$, then there exists $\eta$ such that $G(\tilde{f}) < \Omega$, a contradiction.
        If $A = 0$ and $B\neq C$, then there also exists $\eta$ that leads to a contradiction.
        The only remaining case is $A = 0$ and $B = C$. In this case, for every admissible choice of $\eta$, we have $G(\tilde{f}) = \Omega$.
        Then we choose $\eta = \min \{1 - f(u), f(v)\}$. Then this contradicts the choice of $f$. This completes the proof. 
    \end{claimproof}

    Let $U = \{u\in V(H): f(u) =1\}$. Then by Claim~\ref{clm:one-exception}, we have 
    \begin{equation}\label{eq:U}
        |U| \geq \sum_{v\in V(H)} f(v) - 1 \geq \rho n. 
    \end{equation}
    By the definition of $(\rho, d)$-denseness, we have 
    \allowdisplaybreaks
    \begin{align}
        \Omega &\geq  k! \sum_{e\in E(K_n^{(k)})} H(e) \prod_{u_i \in e} f(u_i) - d\left( \sum_{v\in V(H)} f(v) \right)^k \nonumber \\
        &\geq k! \sum_{e\in E(K_n^{(k)}[U])} H(e)  - d\left( \sum_{v\in V(H)} f(v) \right)^k \nonumber \\
        &= t(K_k^{(k)}, H[U]) - d\left( \sum_{v\in V(H)} f(v) \right)^k \nonumber \\
        &\geq d |U|^k -  d\left( \sum_{v\in V(H)} f(v) \right)^k \label{eq:using-knrs} \\
        &\geq d|U|^k - d(|U| + 1)^k \label{eq:using-U} \\
        &\geq -2^k n^{k-1}. \nonumber
    \end{align}
    The inequality \eqref{eq:using-knrs} holds as $H$ is $(\rho, d)$-dense together with \eqref{eq:U}. Also, \eqref{eq:U} implies \eqref{eq:using-U} holds.
    This completes the proof.
\end{proof}


\end{document}